\documentclass[12pt,a4paper]{article}

\usepackage[latin1]{inputenc}
\usepackage{amsmath}
\usepackage{amsfonts}
\usepackage{amssymb}
\usepackage{makeidx}
\usepackage{graphicx}
\usepackage{listings}

\newtheorem{theorem}{Theorem}[section]

\newenvironment{proof}[1][Proof]{\begin{trivlist}
\item[\hskip \labelsep {\bfseries #1}]}{\end{trivlist}}
\newenvironment{definition}[1][Definition]{\begin{trivlist}
\item[\hskip \labelsep {\bfseries #1}]}{\end{trivlist}}

\newcommand{\qed}{\nobreak \ifvmode \relax \else
      \ifdim\lastskip<1.5em \hskip-\lastskip
      \hskip1.5em plus0em minus0.5em \fi \nobreak
      \vrule height0.75em width0.5em depth0.25em\fi}

\title{Signatures of monic polynomials.}
\author{Norbert A'Campo}

\begin{document}

\maketitle

\begin{abstract}
Let $P:\mathbb{C}\to \mathbb{C}$ be a monic polynomial map of degree $d\geq 1$.  We call  
the inverse image of the union of the real and imaginary axis the geometric picture of the polynomial $P$. The geometric picture of a monic polynomial is a piece-wise smooth planar graph. Smooth isotopy classes relative to the $4d$  
 asymptotic ends at infinity of geometric pictures are called signatures. The set of signatures $\Sigma_d$ of monic degree $d$ polynomials is finite. We give a combinatorial characterization of the set of signatures $\Sigma_d$ and prove that the space of monic polynomials of given signature is contractible.  This construction leads to a real semi-algebraic cell decomposition 
 $$
 {\rm Pol}_d=\bigcup_{\sigma \in \Sigma_d} \{P \mid \sigma(P)=\sigma \}
 $$
 of the space ${\rm Pol}_d$ of monic polynomials of degree $d$.
 In this cell decomposition the classical discriminant locus $\Delta_d$ appears as a union of cells. 
 The complement of the classical discriminant $B_d:={\rm Pol}_d\setminus \Delta_d$ is a union of cells. The face operators of this cell decomposition of   the space $B_d$ are explicitly given. Since $B_d$ is a classifying space for the braid group, we obtain a finite complex that computes the 
group cohomology of the braid group with integral coefficients.  
\end{abstract} 

\section{Introduction.}

 Let $P:\mathbb{C} \to \mathbb{C}$ be a polynomial mapping. We assume that $P$ is monic, i.e. with  leading coefficient $1$. We call a polynomial $P$  balanced if its sub-leading coefficient vanishes which says that 
 the sum of its roots $P^{-1}(0)$ weighted by multiplicity equals $0$. A unique Tschirnhausen substitution $z=z-t$ will transform a 
 monic polynomial to a monic and balanced one. 
We call  the inverse image by the map $P$ 
of the union of the real and the imaginary axis the geometric picture $\pi_P$ of a monic polynomial $P$.

Geometric pictures of monic polynomials are
special graphs in the Gaussian plane $\mathbb{C}$. Their combinatorial restrictions are listed in the following statement.

 \begin{theorem}
 Let $P(z)$ be a monic polynomial of degree $d>0$. Its geometric picture $\pi_P$ is a smooth graph in $\mathbb{C}$
with the following properties:

1.  The graph has no cycles. The graph is a forest. The non-compact edges are properly embedded in $\mathbb{C}$.

2.  The complementary regions  have a $4$-colouring by symbols $A,B,C,D$.

3. The edges are oriented smooth curves and have a $2$-colouring  by symbols $R,I$. They carry the  symbol $R$ if the edge
separates $D$ and $A$ or $B$ and $C$ coloured regions.  They carry the  symbol $I$ if the edge
separates $A$ and $B$ or $C$ and $D$ coloured regions. The orientation is right-handed if one crosses the edge from
$D$ to $A$ or $A$ to $B$, and left-handed if one crosses $B$ to $C$ or $C$ to $D$.  

4. The picture has $4d$ edges that near infinity are asymptotic to the 
rays $re^{k\pi i/2d},\,r>0,\,k=0,1, \cdots ,4d-1$. 
The colours $R,I$ alternate and the orientations
of the $R$ coloured and also the $I$ coloured alternate between out-going and in-going.

5. Near infinity the sectors are coloured in the counter-clockwise orientation by the $4$-periodic sequence of symbols $A,B,C,D,A,B, \cdots $.

6. The graph can have $5$ types of vertices:
for the first $4$ types only $A,B$ or $B,C$ or $C,D$ or $D,A$ regions are incident and only edges of one color are incident, moreover for the fifth type, regions of all $4$ colours are incident and the colours appear in the counter clock-wise orientation as $A,B,C,D,A,B, \cdots $. So, in particular the graph has no terminal vertices.

7. At all points $p\in \pi_P$ the germ of the graph $\pi_P$ is
smoothly diffeomorphic to the germ at $0\in \mathbb{C}$ of $\{z \in \mathbb{C} \mid {\rm Re}z^k=0\}$ for some $k=1,2, \cdots$. 
 \end{theorem}
 
\begin{proof}
The real and imaginary axis decompose the complex plane $\mathbb{C}$ 
in four regions coloured by $A,B,C,D$ according to 
the signs of the real part and the imaginary part  respectively $++,+-,--,-+$. The real and the imaginary axis are coloured
by $R,I$ and are oriented by the gradients of the real and the imaginary part. In fact, this is the colouring and the 
orientation of the picture $\pi_z$ of the degree $1$ polynomial $z$. 
The picture $\pi_P$ inherites the colouring for its regions and the colouring together with orientation of its edges by pulling back via the map $P$ the colouring and orientation of $\pi_z$.
Properties $2,3, \cdots 7$ are clear. 

For property $1$ we first observe that the function ${\bf Re}(P)*{\bf Im}(P):\mathbb{C}\to \mathbb{R}$ is harmonic. Indeed, ${\bf Re}(P)*{\bf Im}(P)=\dfrac{1}{2}{\bf Im}(P^2)$ and the imaginary part of the holomorphic map $P^2$ is harmonic. A minimal cycle $Z$  
in $\pi_P$ is a simply closed curve and 
would bound an  open bounded region $U$. Since   the function ${\bf Re}(P)*{\bf Im}(P)$ vanishes along $Z$, we would 
have   ${\bf Re}(P)*{\bf Im}(P)=0$ on $U$.
It follows that the image $P(U)$ of $P={\bf Re}(P)+i{\bf Im}(P)$ is contained in the union of the real and the imaginary axis in $\mathbb{C}$, contradicting the openness of the non constant holomorphic mapping $P$.  
\qed
\end{proof}

\begin{theorem}
For given degree $d$, there exist  only finitely many  isotopy classes of graphs satisfying the $7$ properties 
of Theorem $1.1$.
\end{theorem}

\begin{proof}
We compactify the graph by adding $4d$ ideal vertices at infinity, one  for each ray, see Property $4$. 
Let $v$ be the number of vertices and $v'$ be the number of inner (non-ideal) vertices of the compactified graph. 
Since the degree of an inner vertex is at least $4$, the number of incidence pairs of a vertex, finite or ideal, 
and edge is at least $4v'+4d$. So the number $e$ of edges is at least $2v'+2d$.
Since the graph is an non-empty forest, its Euler number is at least $1$. Hence
$$1\leq v-e \leq (v'+4d)-(2v'+2d)$$
showing $v'\leq 2d-1$ and $v\leq 6d-1$. The statement follows, since the number of isotopy classes, relative to infinity, of planar forests with $4d$ ideal fixed terminal vertices at infinity and at most
$2d-1$ finite vertices is finite.
 \qed
\end{proof}

\begin{definition}
A signature of degree $d$ is a smooth isotopy class of graphs that satisfy the $7$ properties of Theorem $1.1$.
\end{definition}

Equipped with the Hausdorff topology of proper closed subsets in $\mathbb{C}$, a signature becomes a topological space of planar graphs. 
Classical theorems of Rheinhold  Baer, David Epstein and Jean Cerf in planar topology tell us that a signature is a contractible space. See 
the thesis of Yves Ladegaillerie for the study of spaces of graphs in surfaces.

The following theorems are the main results.  Every signature is realized by a monic polynomial. 

\begin{theorem}
 Let $\sigma$ be a signature of degree $d>0$. Then there exists some  $P\in {\rm Pol}_d$ whose geometric picture belongs to $\sigma$.
\end{theorem}

The space of monic polynomials with given signature is contractible.

\begin{theorem}
Let $\sigma$ be a signature of degree $d>0$. The space $\{P\in {Pol}_d \mid \sigma(P)=\sigma\}$ is contractible.
\end{theorem}

\section{Bi-regular polynomials.}
As intermezzo we first study most generic monic polynomials. The corresponding cells are the open cells.
We call the map
$P$ \textit{bi-regular} if $0 \in \mathbb{R}$ is a
regular value for both mappings, the real aswell as the imaginary part $\textbf{Re}(P):\mathbb{C} \to \mathbb{R}$
and $\textbf{Im}(P):\mathbb{C}  \to \mathbb{R}$. The geometric picture of a bi-regular polynomial $P$ of degree $d>0$ is the union of the
oriented inverse images of $0\in \mathbb{R}$ for
the maps $\textbf{Re}(P)$ and $\textbf{Im}(P)$. It has $d$ vertices of valency $4$ at the roots of $P$ and $4d$ non compact terminal edges. Here we show as an example the geometric picture of  the bi-regular polynomial
$$P=z^{13}-6z^7+z^4-z^3+5z^2+z+3+2i$$
See the picture in Fig. $1$ that we have made with SAGE. The blue lines are the inverse image by $P$ of the oriented
(from $-\infty$ to $+\infty$) real axis, and are also the inverse image by $\textbf{Im}(P)$ of $0\in \mathbb{R}$.
The green ones are the inverse image by $P$ of the oriented (from $-i\infty$ to $+i\infty$) imaginary axis and also
the inverse image by $\textbf{Re}(P)$ of $0\in \mathbb{R}$. There are
$13$ transversal intersection points of a green and blue line, which of course are the roots of the
polynomial $P$. At each root a blue and a green line intersect orthogonally since the polynomial map $P$ is conformal and hence its differential at regular points preserves angles. Each blue or green line
is an properly embedded copy of the oriented real line in the plane. Near infinity those lines are
asymptotic to rays emanating from the origin.
 For a bi-regular polynomial of degree $d$, the inverse image of the real axis is a disjoint union of $d$ copies of an oriented  real line, having $2d$ ends that are asymptotic to
rays directed by the $2d$-roots of unity $\theta$ with $\theta^d=\pm 1$. The inverse image of the imaginary axis is a similar disjoint union of $d$ copies of a real line, except that
the ends are asymptotic to the directions of the $4d$-roots of unity $\theta$ with $\theta^d=\pm i$. We orient the asymptotic rays from $0$ to $\infty$. The
orientation of a curve of the picture and its asymptotic ray match if $\theta^d=+1,+i$ and are opposite if $\theta^d=-1,-i$.  We call the geometric picture of a bi-regular polynomial a bi-regular picture.

We say that two bi-regular pictures $\pi,\pi'$ are combinatorially equivalent if there exists a regular proper ambient isotopy that keeps the direction of the asymptotics fixed and that
moves $\pi$ to $\pi'$.  A bi-regular signature is a combinatorial equivalence class of bi-regular pictures.

In the next section we will count  the number of bi-regular signatures 
of degree $d$. A bi-regular signature is a signature such that every vertex has valence $4$ at which 
the incident $4$ sectors have $4$ different colours.

This is especially interesting for following a root $r_t$ continuously
given  by a family $P_t$ of polynomials. What is still missing, is an understanding of the wall crossings phenomena  between different connected components of bi-regular polynomials.  In particular we do in particular not know the dual graph of those components 
for which a component becomes a vertex and a pair of vertices  is connected by an edge if  one gets from one component to the other by a transversal wall crossing. We plan applications to computer graphics and robotics in future.
\bigskip

\includegraphics[scale=0.7]{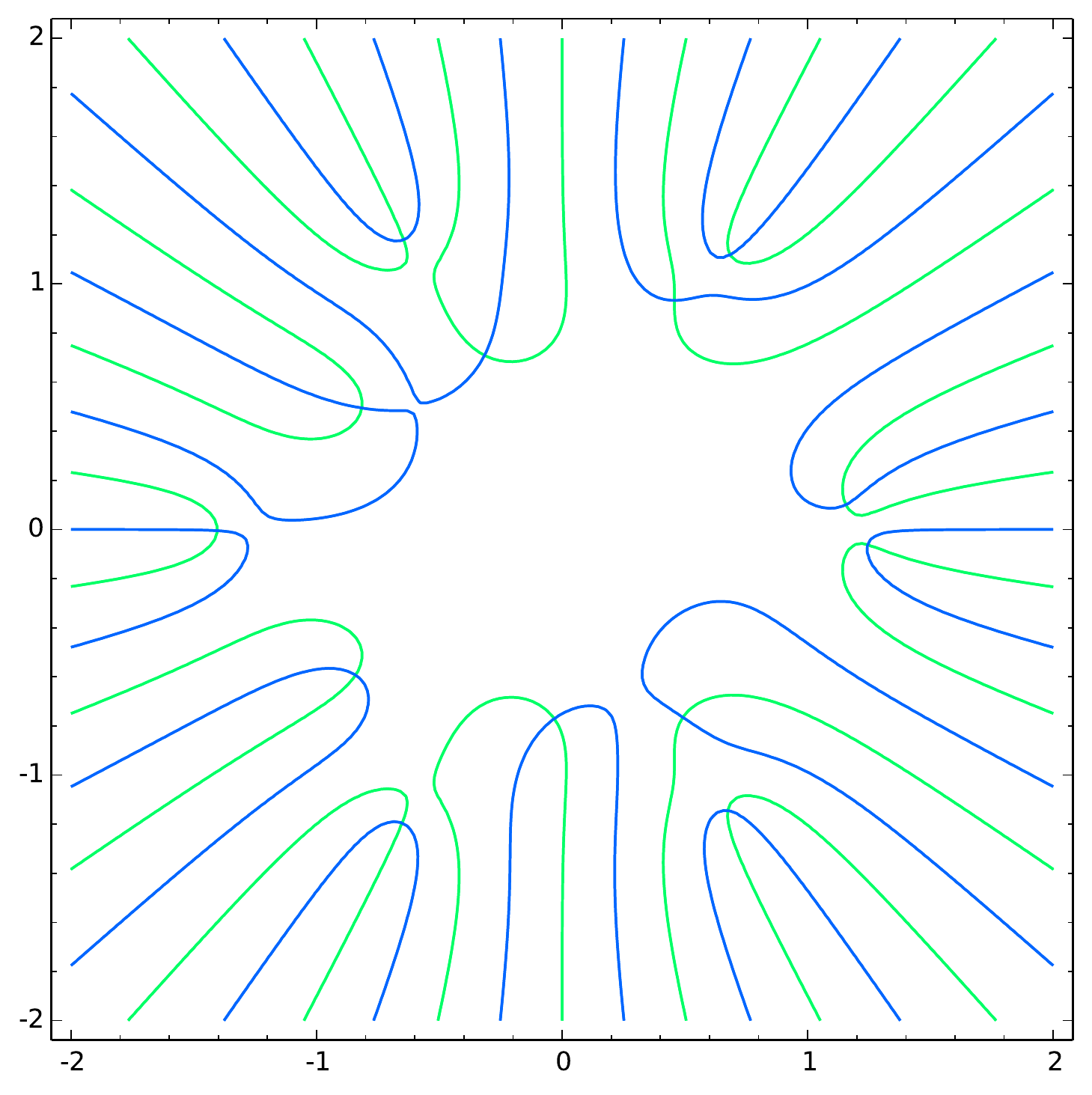}

\centerline{Fig. $1$. $P=z^{13}-6z^7+z^4-z^3+5z^2+z+3+2i$.}

\section{Counting bi-regular and sub bi-regular signatures.}
\noindent

Let $P:\mathbb{C} \to \mathbb{C}$ be a monic polynomial mapping of degree $d$.  We assume that $0\in\mathbb{R}$ is a regular value of the imaginary part mapping of ${\bf Im}(P):\mathbb{C}\to \mathbb{R}$.
The inverse image $P^{-1}(\mathbb{R}) \subset \mathbb{C}$ is a system of $d$ smoothly embedded copies of the real line. The orientations of this systeme can be reconstructed,
since the positive end of the real axis is an asymptotic ray with matching orientations and since the matching and non-matching ends alternate if one goes from one
$2d$-root of unity to the next. So we can forget the orientations of the components of   $P^{-1}(\mathbb{R}) \subset \mathbb{C}$ without loosing information.
Combinatorially we can think of $P^{-1}(\mathbb{R}) \subset \mathbb{C}$ as a system of $d$ disjoint diagonals and edges in a $2d$-gon. The number of possible systems $D(d)$
of $d$ non intersecting
diagonals or edges in a  $2d$- gon  is given by a Catalan number. We put $D(0)=1$, and  have $D(1)=1,\,D(2)=2$ and for $d\geq 3$
of $d$ non intersecting
diagonals or edges in a  $2d$- gon  is given by a Catalan number. Here in particular \textit{Disjoint} means having no common vertices!  Moreover, for $d\geq 3$
the recurrence relation
$$
D(d+1)=\sum_{0\leq i \leq d} D(i)D(d-i)
$$
holds.
This recurrence relation is obtained by splitting a $2(d+1)$-gon along the curve that has the first vertex as end.
This is the recurrence relation for the Catalan numbers, hence
$D(d)=\frac{1}{d+1}{2d \choose d}$. 
The first Catalan numbers for $d=1,2, \cdots$ are
$$
1,2,5,14,42,132,429,1430
$$
For the inverse image of the imaginary axis we also have  $D(d)$ possibilities. The two possibilities are very depending, since each
component of the inverse image of the real axis intersects  transversally the inverse image of  the imaginary axis once. So we need a combined counting.
Let ${\rm Pict}(d)$ be the number of possible combinatorial types of pictures. We put ${\rm Pict}(0)=1$ and have ${\rm Pict}(1)=1$.
For $d\geq 2$ we have the recurrence relation
$$
{\rm Pict}(d+1)=\sum_{0\leq i,0\leq j,0\leq k,0\leq l, i+j+k+l=d} {\rm Pict}(i){\rm Pict}(j){\rm Pict}(k){\rm Pict}(l)
$$
which is obtained from the following splitting: let $A$ be the curve in $P^{-1}(\mathbb{R})$ that is asymptotic to the positive real axis and $B$ be
the curve in $P^{-1}(i\mathbb{R})$ that intersects $A$.  The pair of curves $(A,B)$ splits the complex plane in four regions. The summing indices  $i,j,k,l$ are the number of roots of a bi-regular polynomial in these regions. The Catalan recurrence expresses $D(d+1)$ as a sum of products $D(a)D(b)$ with $a+b=d$. The recurrence for
${\rm Pict}(d+1)$ is similar, except that ${\rm Pict}(d+1)$ is a sum of $4$-factor products.
In order to integrate this recursion we first computed with a PARI  program the first $15$ terms. The result was
$$
1,4,22,140,969,7084,53820,420732,3362260,27343888,
$$
$$
225568798, 1882933364, 15875338990, 134993766600, 1156393243320
$$
A search in the on-line Encyclopedia of Integral Sequences founded in 1964 by N.J.A. Sloane, see \textit{https:/oeis.org/}  identifies this sequence
with the sequence $A002293$ and shows to us many interesting interpretations.
Also we learn, that the closed formula is of Fuss-Catalan type:
$$
{\rm Pict}(d)=\frac{1}{3d+1}{4d \choose d}
$$
By induction upon $d$ we check that the proposed expression satisfies the recursion relation of ${\rm Pict}(d)$.

The space of bi-regular polynomials of degree $d$ is an open subset in the space of all degree $d$ polynomials. 
\bigskip

 \includegraphics[scale=0.7]{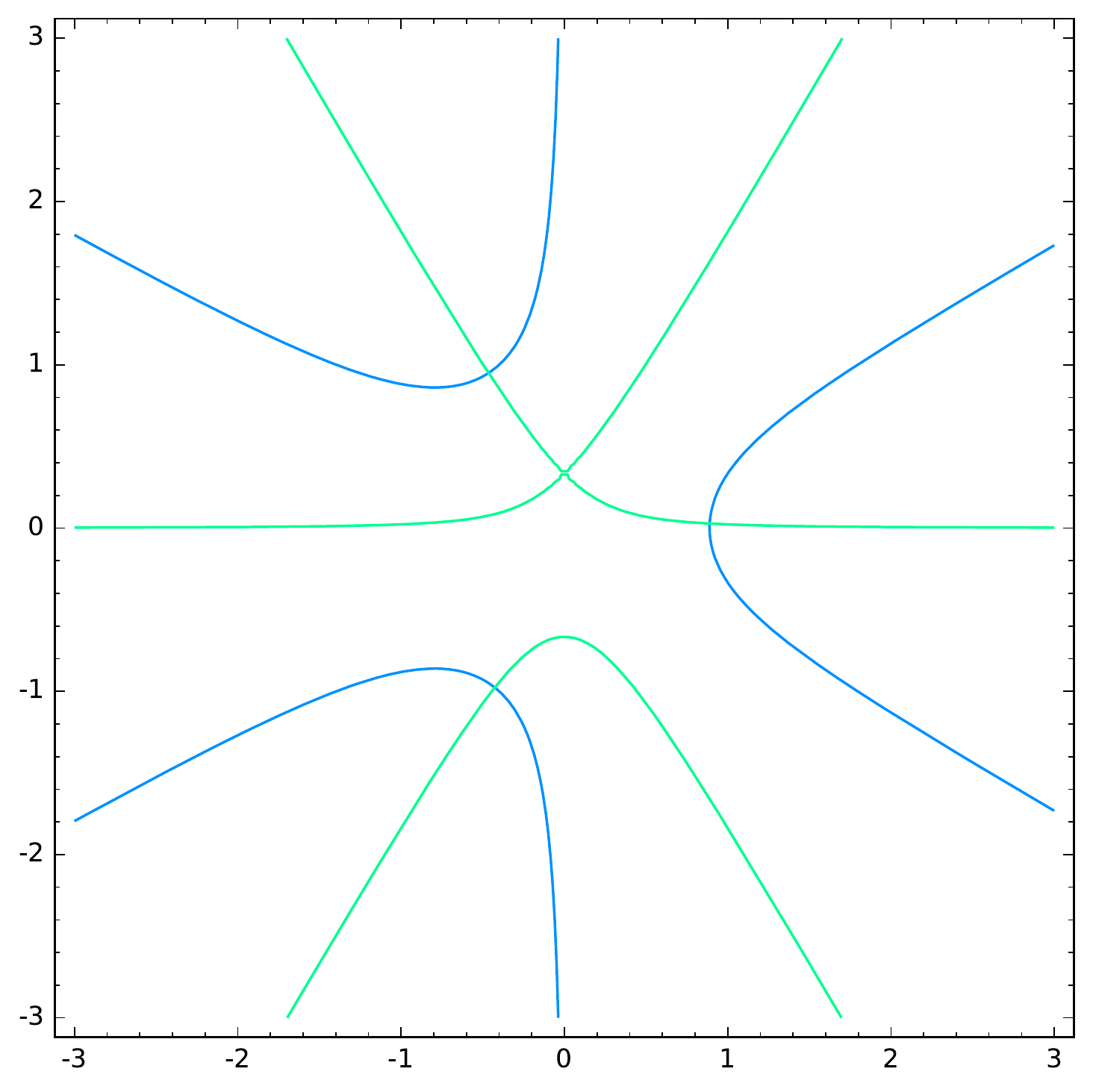}
 
  \centerline{Fig. $2$. $P(z)=(z-i/3)^3+(z-i/3)^2+1$.}
\bigskip  

The connected components
correspond bijectively to pictures of bi-regular polynomials. We say that two components are neighbours if separated by a wall of real co-dimension $1$. In this case we also say that two bi-regular pictures or bi-regular signaturs  are neighbours. The signature of the picture of Fig. $2$ defines such wall that separates two bi-regular components. The picture of Fig. $2$ allows two smoothings that yield bi-regular pictures.

The discriminant $\Delta\subset \mathbb{C}^d$ is the space of monic polynomial mappings $P:\mathbb{C} \to \mathbb{C}$ having $0\in \mathbb{C}$ as critical value. Clearly, a polynomial $P$ belongs to $\Delta$ if and only if the mappings $\textbf{Re}(P):\mathbb{C} \to \mathbb{R}$
and $\textbf{Im}(P):\mathbb{C}  \to \mathbb{R}$ have a critical point with critical value $0$ in common. It follows that each cell of bi-regular polynomials is contained in the complement of $\Delta$.
 It also follows that  the co-dimension $1$ walls are contained in the complement of $\Delta$.
 
 The polynomial $P(z)=z^3-z/3+28/27$  is regular above $0$ as mapping from $\mathbb{C}$ to $\mathbb{C}$, but
 the map $\textbf{Re}(P):\mathbb{C} \to \mathbb{R}$ has two critical points with $0$ as value. The polynomial $P$ belongs to a stratum of real co-dimension $2$. See Fig. $3$. The two critical points of $\textbf{Re}(P):\mathbb{C} \to \mathbb{R}$ can fuse together in a stratum of real co-dimension $3$. See the picture of $Q(z)=z^3+1$ in Fig. $4$. 
\bigskip

 \includegraphics[scale=0.7]{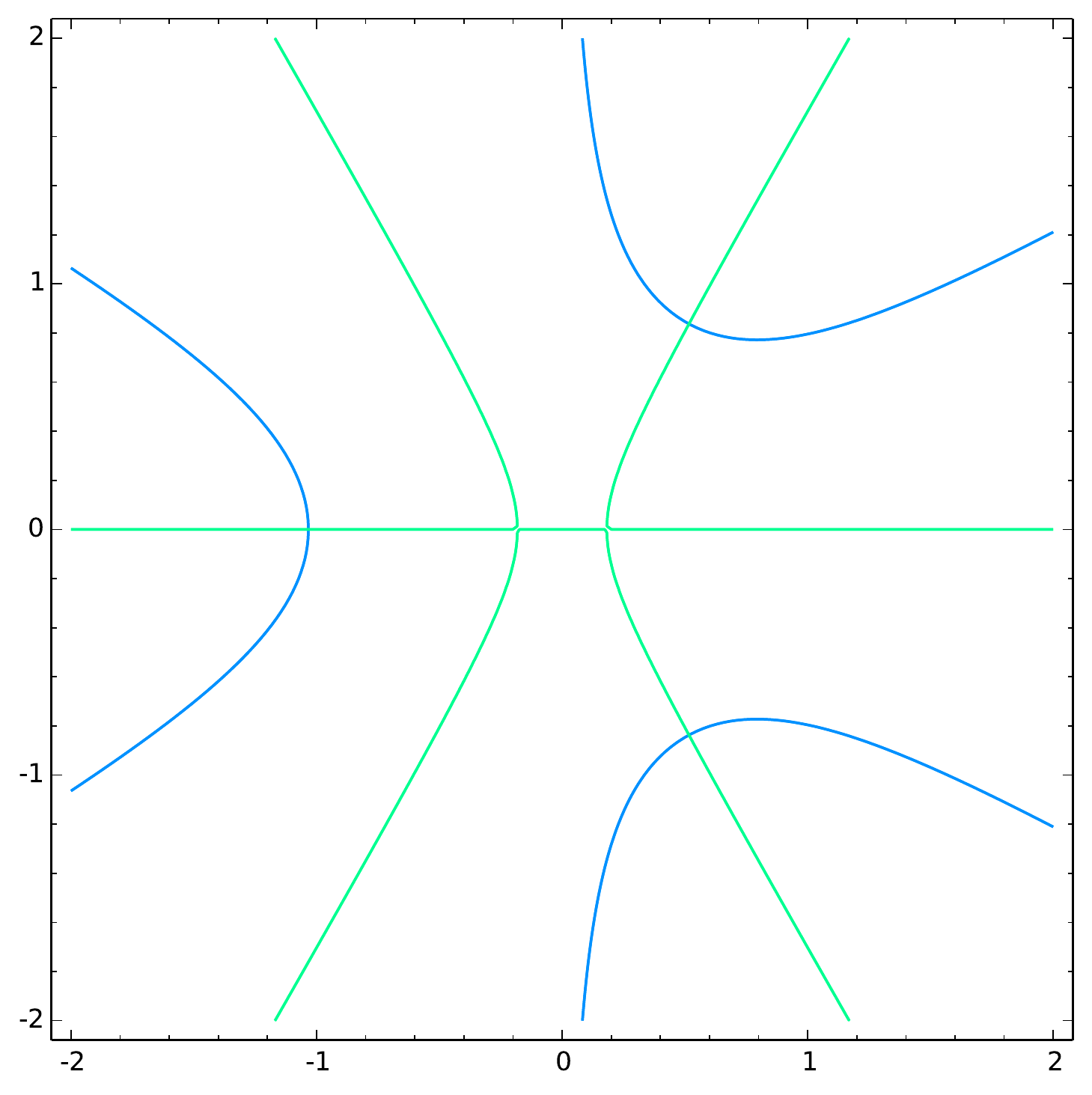}
 
  \centerline{Fig. $3$. $P(z)=z^3-\dfrac{1}{10}z+1$.}

\bigskip

 The polynomial $P(z)=z^3-\dfrac{1}{10}z+1$  is regular above $0$ as mapping from $\mathbb{C}$ to $\mathbb{C}$, but
 the map $\textbf{Re}(P):\mathbb{C} \to \mathbb{R}$ has two critical points with $0$ as value. The polynomial $P$ belongs to a stratum of real co-dimension $2$. See Fig. $3$. The two critical points of $\textbf{Re}(P):\mathbb{C} \to \mathbb{R}$ can fuse together in a stratum of real co-dimension $3$. See the picture of $Q(z)=z^3+1$ in Fig. $4$.
 
The two critical points of $\textbf{Re}(P):\mathbb{C} \to \mathbb{R}$
can be smoothed, one a lot, the other less, see Fig. $5$. 

Again with a PARI program we could compute the numbers of co-dimension $1$  walls in degree $d=1,2,3, ...$ .
We get:
$$ 0, 4, 48, 480, 4560, 42504, 393120, 3624768, 33390720, 307618740$$
This sequence is not identified by the Sloane data base.
\bigskip

\noindent{\bf Problems.} Let $B(d,c)$ be the number of cells in $B_d$ of codimension $c$. Study the generating series
$$C(x,y)=\sum_{d,c} B(d,c) x^dy^c\in \mathbb{Z}[[x,y]]$$ 
and the coefficients
$$C_c(x)=\sum_d B(d,c)x^d\in \mathbb{Z}[[x]]$$
Study the differential operators that annilate $C(x,y),\,C_c(x)$. Find closed expressions for $B(d,c)$. 

\bigskip

 \includegraphics[scale=0.6]{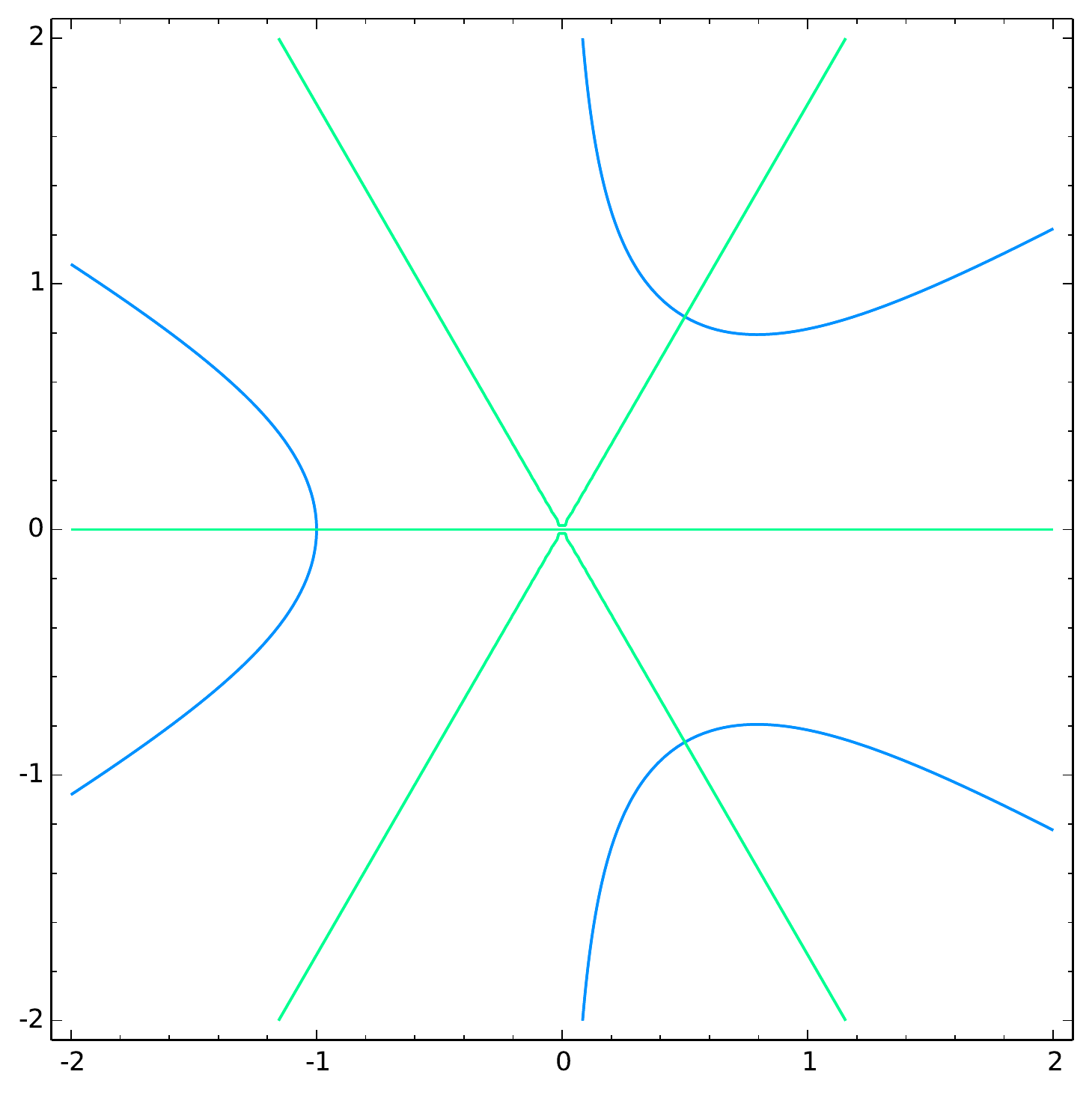}

\centerline{Fig. $4$. $Q(z)=z^3+1$.}

 \includegraphics[scale=0.6]{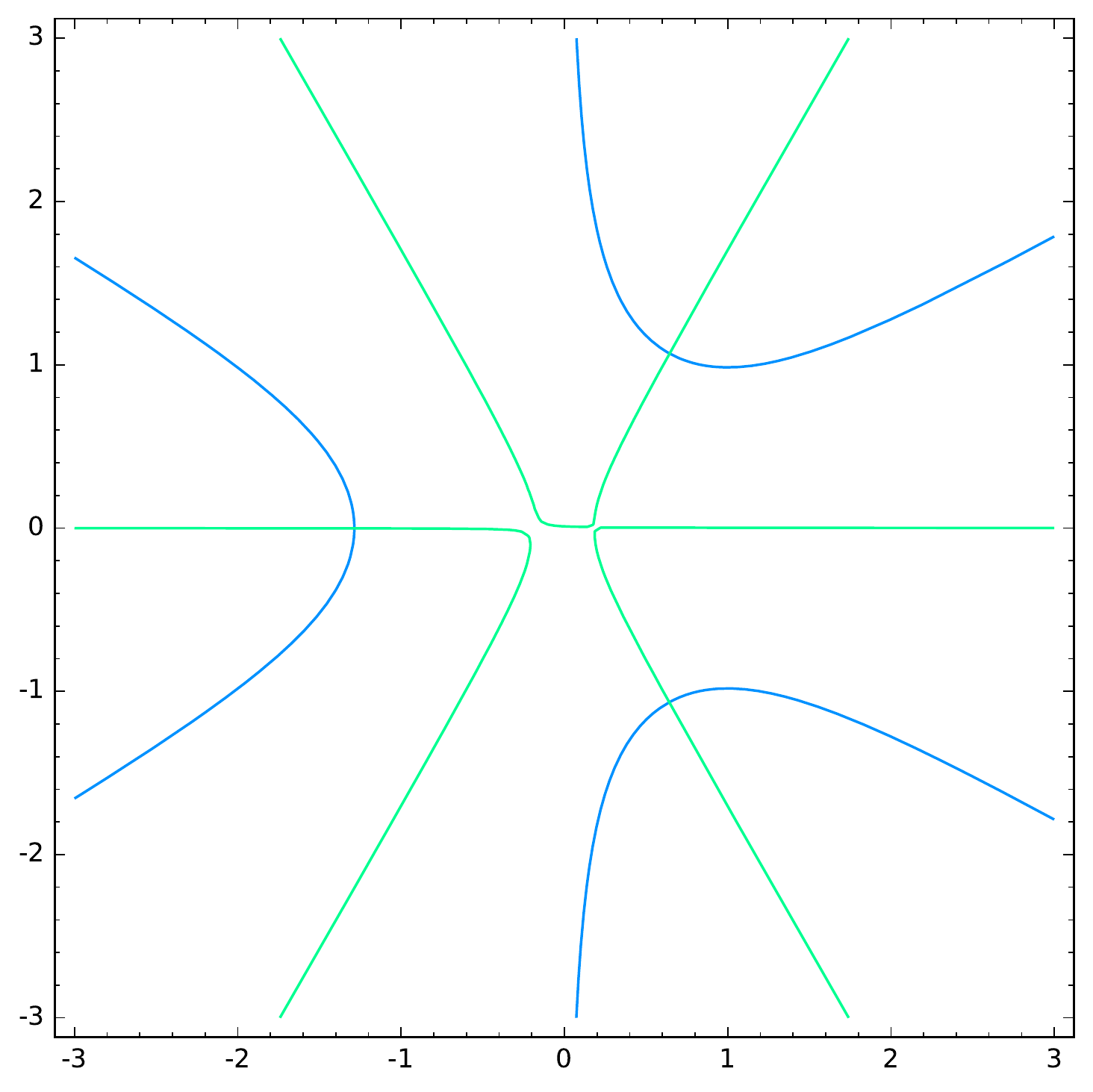}

\centerline{Fig. $5$. $P(z)=z^3-(\dfrac{1}{10}+\dfrac{i}{200})z+1+\dfrac{i}{1000}$.}

\bigskip

 \includegraphics[scale=0.6]{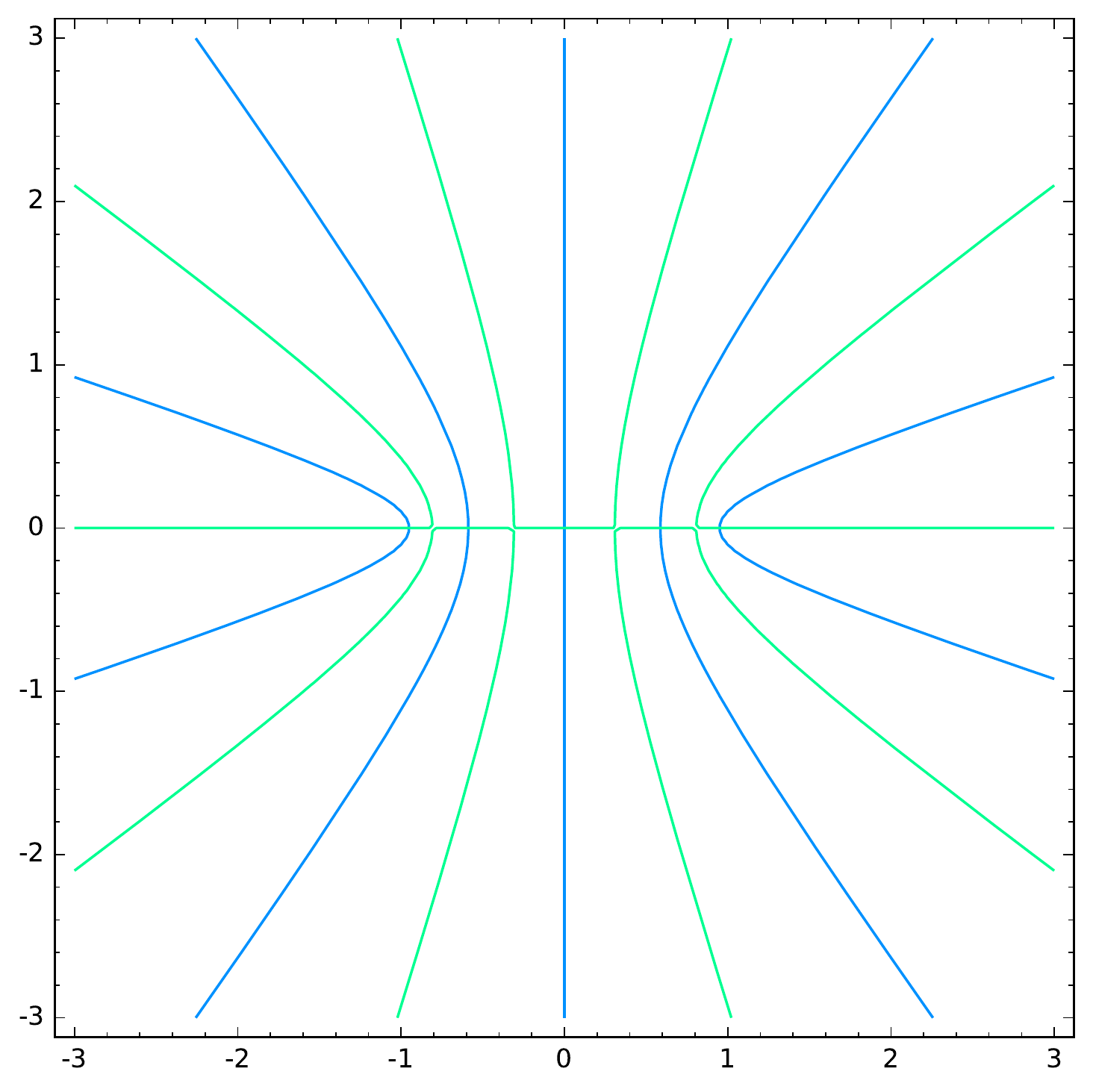}

\centerline{Fig. $6$. $T(5,z)=16z^5-20z^3+5z$.}

Special polynomials have typical pictures. As example see the fifth
Chebyshev polynomial of the first kind in Fig. $6$. One observes that its picture can be smoothed at $4$ places. So, the fifth Chebyshev polynomial
belongs to a cell of codimension $4$. This cell is in the closure of 
$2^4$ cells of bi-regular polynomials. This holds for all degrees: the
Chebyshev polynomial $T_n$ of degree $n$ belongs to a cell of codimension $n-1$ along which $2^{n-1}$ bi-regular cells meet. Incidently, observe that $2^{n-1}$ is the leading coefficient of the polynomial $T_n$. The cell of the signature $\sigma(T_n)$ is the space of all real monic Morse deformations of the polynomial $z^n$ with $n-1$ real critical points. 

\section{Proofs.}

The proofs are based on the Riemann Mapping Theorem in combination with theorems of Reinhold Baer, David Epstein and Jean Cerf
on homotopy versus isotopy and theorems of C.J. Earle and J. Eells on contractability of connected components of groups of diffeomorphism in dimension two.

\begin{proof} {\bf of Theorem 1.3.}
Let $\sigma$ be a signature and let $\gamma$ be a smooth oriented, coloured embedded graph in the class $\sigma$. Let $4d$ be the number of ideal vertices. The $7$ properties allow to construct a smooth function $f:\mathbb{C} \to \mathbb{C}$ such that the following holds.

1. The graph $\gamma$ is the inverse image by $f$ of the union of the real and the imaginary axis.

2. The map $f$ is open with at most  $d-1$ critical points . The determinant of the tangent map $Df$  is positive at all regular points of $f$. At each critical point of $f$ the the germ of $f$ is smoothly equivalent to the germ of $z\in \mathbb{C} \mapsto z^k+t\in \mathbb{C}$ for some $k=1,2, \cdots $ and some $t\in \{+1,-1,+i,-i,0\}$. 

3. The restriction of $f$ to an edge of $\gamma$ is regular and injective.

4. The colourings of regions and edges of $\gamma$ are the pull-backs by $f$ of the  the colourings of $P_z$.

5.   $\lim_{z\in \mathtt{C}, |z|\to +\infty} \dfrac{f(z)}{z^d}=1$.

Let $J$ be the pull-back by $f$ of the standard conformal structure $J_0$ on $\mathbb{C}$ to $\mathbb{C}$. The map $f:(\mathbb{C},J)\to (\mathbb{C},J_0)$ is holomorphic. By the Riemann
mapping theorem a biholomorphic map $\rho:(\mathbb{C},J)\to (\mathbb{C},J_0)$ exists. Indeed, by property $5$ for $f$, the map extends to a self-map of $\mathbb{C}\cup \{\infty\}$. 
Replacing $\rho$ finally by a positive real multiple $\lambda \rho$ the composition $f\circ \rho:\mathbb{C}\to \mathbb{C}$
by Rouch\'e's Theorem will be a monic polymial having a picture in the class of $\gamma$.
\qed
\end{proof}

\begin{proof} {\bf Theorem 1.4.} Consider the signature $\sigma$ as a space $\Gamma$ of smooth oriented planar graphs.
The space $\Gamma$, if equipped with the topology induced by the topology of the oriented arc-length  parametrizations of the edges, is contractible
by Theorems of Baer and Epstein. Given 
$\gamma\in \Gamma$, the space $E_{\gamma}$ with the smooth topology  of functions $f:\mathbb{C}\to \mathbb{C}$ satisfying the $5$ properties as in the previous Theorem is contractible. The space $E_{\Gamma}$ 
of pairs $(f,\gamma)$ with
$\gamma \in \Gamma$ and $f\in E_{\gamma}$ by a Theorem of J. Cerf is the total space of a fiber bundle
$\pi:E_{\Gamma} \to \Gamma,\, (f,\gamma) \mapsto \gamma$. It follows that the space $E_{\Gamma}$ is contractible.

The group $G_{\mathbb{C},\infty}$ of orientation preserving diffeomorphisms of $\mathbb{C}$ extending to $\mathbb{C}\cup \{\infty\}$ as a diffeomorphism with the identity as differential at $\infty$, is contractible. The group $G_{\mathbb{C},\infty}$ acts with closed orbits and without fixed points
on $E_{\Gamma}$. So the space $E_{\Gamma}/G_{\mathbb{C},\infty}$ is contractible. By the Riemann mapping Theorem, there exists in every $G_{\mathbb{C},\infty}$-orbit a unique pair  $(f,\gamma)$ such that the pull back by $f$ of the standard conformal structure on 
$ \mathbb{C}$ is again the  standard structure. In order to achieve uniqueness of the pair $(f,\gamma)$, we require moreover that
$f$, now by Rouch\'e's Theorem a monic polynomial, is balanced. It follows that the space of monic and balanced polynomials with picture in $\Gamma$ is a space of representatives for the quotient 
$E_{\Gamma}/G_{\mathbb{C},\infty}$. 

We conclude that the space of monic balanced polynomial mappings $P$ with picture in the isotopy class $\Gamma$  is contractible. The group of Tschirnhausen substitutions is contractible and acts fixed point free on the space of monic polynomial mappings $P$ with picture in the isotopy class $\Gamma$. Hence, the space of monic polynomial mappings $P$ with signature $\sigma$  is contractible too.
\qed
\end{proof}

{\bf Labelling roots.} Let $r$ be a root of a monic polynomial $P$. 
The root $r$ belongs to a connected component $T$ 
of the picture of $P$. The component $T$ is an coloured oriented tree. We define as label the pair of $(\alpha,\beta)$ consisting of the
$4d$-root of unity. The root of unity  $\alpha=e^{2\pi i k}/2d$, with $k\in \{0,1, \cdots ,2d-1\}$ minimal,  is in fact the $2d$-root of unity, that we get by starting at $r$ and by following in $T$ the oriented edges in $P^{-1}(\mathbb{R})$. The root of unity $\beta=e^{2\pi i (2k+1)/4d}$ with minimal $k\in \{0,1, \cdots ,2d-1\}$ is the root of unity that we get by starting at $r$ and by following in $T$ the oriented edges in $P^{-1}(i\mathbb{R})$. 

Essentially, from its label we can find back the corresponding root by solving differential equations. The map from root to label is constant in each cell of the cell decomposition by signatures.  This property has clearly applications, each time one wishes to follow roots of polynomials continuously in families of polynomials. Robotics typically encounters this wish.

\section{Pictures of meromorphic functions.} The real axis $\mathbb{R}\cup \{\infty\}$ and the imaginary axis $i\mathbb{R}\cup \{\infty\}$, both extended by the point $\infty$, divide the Riemann sphere $P^1(\mathbb{C})=\mathbb{C}\cup \{\infty\}$ in four regions, that again we label by the colours $A,B,C,D$. Define the picture of a rational map $f: P^1(\mathbb{C})\to P^1(\mathbb{C})$ be the inverse image of the union of the extended axis. Similarly, define the picture of a holomorphic map $f:S\to P^1(\mathbb{C})$ on a Riemann surface. Call a rational map or more general a meromorphic function $f$ on a Riemann surface very bi-regular if the critical values do not belong to the extended real or the extended imaginary axis. Call a function $f$ bi-regular if its critical values do not belong to the (non extended) real or imaginary axis. The bi-regular polynomials remain according to this definition bi-regular. 
We plan to study in future from the combinatorial view point these more general settings.   

\section{Face operations, remarks, questions.}  The most non-bi-regular polynomial are $P=(z-r)^d, d\geq 2,\,r\in \mathbb{C}$. Each connected component of bi-regular polynomials of degree $d>1$ has such a  polynomial in its closure. Indeed, let $P$ be bi-regular of degree $d>1$. The family of bi-regular polynomials $t^dP(z/t),\, t\in \mathbb{R},\ t>0,$ has as limit at $t=0$
the polynomial $z^d$.
The  family is an orbit of the weighted homogeneous action of the group of positive real numbers
$$(t,P) \mapsto t\bullet P=t^{{\rm degree}(P)}P(z/t)$$  
In fact the larger group of affine substitutions 
$$z \mapsto z/t+a,\,t>0,\,a\in \mathbb{C}$$
acts on monic polynomials.
We can use this action for simplification and normalization. We define as norm $\Vert P \Vert$ the $l^2$-norm of the vector of coefficients of $P$.
First, for a polynomial $P\not=(z-r)^d$ of degree $d>1$ we simplify 
by a Tchirnhausen substitution $z \mapsto z-a$ killing  the coefficient of the term $z^{d-1}$ in $P$,
next we choose $t>0$ in order to get a polynomial of norm $1$.
This is possible since for $P\not= z^d$ there exists precisely one $t>0$ with $\Vert t\bullet  P\Vert=1$.
So instead of studying the chambers of bi-regular polynomials in the vector space of complex dimension $d$ of all monic polynomials of degree $d$, one can restrict this study to the unit sphere of dimension $2d-3$
in the space of degree $d$ Tchirnhausen simplified polynomials. For instance, it would be interesting to study what happens for $d=3$. One gets a decomposition of the sphere $S^3$ in $22$ contractible components. What is the dual graph?

Face operators correspond to the following two operations on signatures. 
Let $\sigma$ be a signature. Let $\pi$ be a picture in the class $\sigma$. The picture $\pi$ decomposes the plane $\mathbb{C}$ in polygonal regions. The regions have piece-wise smooth curves as boundaries. 

The first operation consists in contracting diagonals of regions. Let $D$ be a smooth generic diagonal in such a region connecting two boundary points, such that the graph $\pi\cup D$ is still a forest. The endpoints of $D$ are smooth points of edges. The new graph $\pi\cup D$ obtained by adding $D$ to the picture $\pi$ does not 
satisfy the $7$ properties. In particular, two vertices are of degree $3$.  Let $\pi_D$ be the planar graph obtained by contracting the diagonal $D$ to a new vertex of degree $4$. The graph $\pi_D$, together with its colouring of edges, labelling of regions,  satisfies the $7$ properties and the class of $\pi_D$ is again a signature $\sigma_D$.  

The second operation consists of contracting an edge of $\pi$ connecting two vertices which are no roots, i.e. vertices not incident with all four colors $A,B,C,D$. Contracting the edge $E$ in $\pi$  transforms $\pi$ to a new picture with the $7$ properties, so constructs a new signature $\sigma_E$. 

The operation of adding and contracting a diagonal 
$D$ to $\sigma$ or the operation of contracting an edge $E$ such that $\sigma_D$ or $\sigma_E$ is again a signature corresponds to a co-dimension one face operation for the cell decomposition of the space $B_d$. 

It is a challenging problem to understand the combinatorics of these face operations and to describe the corresponding cell and co-chain complex for the spaces $B_d$.  

The semi-algebraic cell decomposition of $B_d$ is compatible with a triangulation by a theorem of Stanis$\l{}$aw $\L{}$ojasiewicz. Working in the second barycentric subdivisions allows to construct regular open neigborhoods $U_{\sigma}$ in $B_d$ of closures in $B_d$ of cells $\{P\mid \sigma(P)=\sigma\}\subset B_d$. The integral $\Check{C}$ech-cohomology of the acyclic covering
 $\{U_{\sigma}\}$ of $B_d$   computes the group cohomology $H^*((Br(d),\mathbb{Z})$ of the braid group $Br(d)$. It is a challenging problem to do this computation effectivily.

A third face operation is needed in spaces of polynomials that have roots of multiplicities exceeding one. Instead of contracting edges,
now also contracting minimal subtrees $T$ with two or more edges in $\pi$  such that 
the class $\sigma_T$ of $\pi/T$  is again a signature gives a face operation.  

The unit sphere $S^{2d-2}$ in the space of complex Tchirnhausen simplified  polynomials of degree $d$ has a natural probability measure. A natural question is about the probability that a random polynomial $P$ realizes a given picture? Which picture has highest probability?

The notion bi-regularity suggests two notions of discriminants for polynomials. We define as real discriminant the set $\Delta_{\mathbb{R}}$ of all monic polynomials $P$
of a given degree
such that $0\in \mathbb{R}$ is a critical value of the real part $\textbf{Re}(P)$, and accordingly $\Delta_{i\mathbb{R}}$ all polynomials with $0\in \mathbb{R}$ as critical value for $\textbf{Im}(P)$.
Recall that the classical discriminant $\Delta$ is the set of polynomials $P$ such that the mapping $P:\mathbb{C} \to \mathbb{C}$ has $0\in \mathbb{C}$ as critical value. The complement of the union $\Delta_{\mathbb{R}} \cup \Delta_{i\mathbb{R}}$ is the set of bi-regular polynomials and the classical discriminant $\Delta$
is included in the intersection $ \Delta_{\mathbb{R}} \cap \Delta_{i\mathbb{R}}$.

We get a braid invariant as follows. A braid $b$ defines  an isotopy class of a  closed path of monic polynomials  in the complement of $\Delta$. What is the minimal number of bi-regular chambers that such a path for a given braid has to visit?
\newline
\noindent{\bf Acknowledgement.} This work started in 2014 
at the Graduate School of Sciences, Hiroshima University during the Conference "Branched Coverings, Degenerations and Related Topics". 
The author thanks Professors Makoto Sakuma and Ichiro Shimada for the  warm hospitality and for providing stimulating mathematical  environment. 

\section{Sage and Pari scripts.}

\begin{lstlisting}
"Hello Pari"  Computes the vector of the number of possible
pictures of  the monic degree deg <= g bi-regular polynomials.

vector_numb_pict(g)=
{
X=vector(g,i,0);
X[1]=1;
for(deg=2,g,
        for(a=0,deg-1,
        for(b=0,deg-1-a,
        for(c=0,deg-1-a-b,
        for(d=deg-1-a-b-c,deg-1-a-b-c,
    X[deg]=X[deg]+
                 if(a==0,1,X[a])*
                        if(b==0,1,X[b])*
                             if(c==0,1,X[c])*
                                    id(d==0,1,X[d]);
              ); ); ); );
   );
  X
}
numb_pict(deg)=
{
X=vector_numb_pict(deg);
X[deg]
}

Computes the number of codimension 1 walls in the
space of monic degree deg polynomials.

numb_wall(deg)=
{
if(deg<3,Res=(deg-1)*4,
      X=vector_numb_pict(deg-1);
      Res=2*deg*sum(a=1,deg-1,X[a]*X[deg-a]);
     );
Res
}
\end{lstlisting}
\begin{lstlisting}
"Hello Sage" Draw the picture of a polynomial. 
Here as example P=z^13-6z^7 .....

import matplotlib
p=Graphics()
var('x','y',domain=RR)
z=x+i*y
f=expand(z^13-6*z^7+z^4-z^3+5*z^2+z+3+2*i)
u=(f+conjugate(f))/2
v=-i*(f-u)
p1=implicit_plot(u==0,(x,-4,4),(y,-4,4),color=rainbow(5)[2])
p2=implicit_plot(v==0,(x,-4,4),(y,-4,4),color=rainbow(5)[3])
p=sum([p1,p2])
p.show()
\end{lstlisting}

\noindent
{\bf References.}

Reinhold Baer, Isotopien von Kurven auf orientierbaren, geschlossenen Fl\"achen. Journal f\"ur die Reine und Angewandte Mathematik 159, 101-116 (1928).

David Epstein, Curves on 2-manifolds and isotopies. Acta Mathematica 115, 83-107 (1966).

Jean Cerf, La stratification naturelle des espaces de fonctions diff\'erentiables r\'eelles et le th\'eor\`eme de la pseudo-isotopie, Inst. Hautes \'Etudes Sci. Publ. Math. No. 39 (1970) 5--173.

Earle, C. J. and  Eells, J., The diffeomorphism group of a compact Riemann surface, Bull. Amer. Math. Soc. 73 (1967), no. 4, 557-559. 

Yves Ladegaillerie,  D\'ecoupes et isotopies de surfaces topologiques,
Th\`ese de Doctorat d'\'Etat (1976), Facult\'e des Sciences, Montpellier.

Stan$\l{}$islaw $\L{}$ojasiewicz, Triangulation of semi-analytic sets. Annali della Scuola Normale Superiore di Pisa - Classe di Scienze 18.4 (1964): 449-474. 

N. J. A. Sloane, A Handbook of Integer Sequences, Academic Press, (1973).

OEIS Foundation Inc. (2011), The On-Line Encyclopedia of Integer Sequences, http://oeis.org.

William A. Stein et al. SAGE Mathematics Software (Version 6.1.1),
   The Sage Development Team, YYYY, http://www.sagemath.org.

PARI/GP, version {\tt 2.5.0}, Bordeaux, (2013), {http://pari.math.u-bordeaux.fr/}.

Louis Comtet, Analyse Combinatoire I, II, Collection SUP, Presses Universitaires de France, (1970).

Reinhold Remmert (1998) Classical topics in complex function theory, Springer-Verlag.

Bernhard Riemann (1851) Grundlagen f\"ur eine allgemeine Theorie der Functionen einer ver\"anderlichen complexen Gr\"osse, G\"ottingen.

\bigskip

\noindent{\rm  University of Basel.}\newline
\noindent{norbert\dots acampo at unibas \dots ch}

\end{document}